\documentclass{amsart}
\usepackage{amsmath, amssymb, amscd, amsbsy, amsthm}
\usepackage{graphicx}
\newtheorem{theorem}{Theorem}[section]
\newtheorem{lemma}[theorem]{Lemma}
\newtheorem{prop}[theorem]{Proposition}

\theoremstyle{definition}

\theoremstyle{remark}
\newtheorem{remark}[theorem]{Remark}
\numberwithin{equation}{section}

\def\N{{\mathbb N}}
\def\Z{{\mathbb Z}}

\def\R{{\mathbb R}}

\def\aread{{\rm Diff}_{\Omega}^{\infty}(D^{2}, \partial D^{2})}
\def\Cal{{\rm Cal}}
\def\Ker{{\rm Ker}}

\newcommand\gmet[1]{\mathcal{M}({#1})}
\def\volm{{\rm Diff}_{\Omega}^{\infty}(M)_{0}}
\def\hamu{{\rm Ham}^{\infty}(U)}
\def\hamr{{\rm Ham}^{\infty}_{C}(\R ^{2n})}
\def\hamd{{\rm Ham}^{\infty}_{C}(D^{2n})}
\def\hamm{{\rm Ham}^{\infty}(M)}

\begin{document}
\title[Quasi-isometry type of  the kernel of the Calabi homomorphism]
{Quasi-isometry type of  the metric space 
derived from the kernel of the Calabi homomorphism}
\author{Tomohiko Ishida}
\address{Department of Mathematics, 
Kyoto University, 
Kitashirakawa Oiwake-cho, Sakyo-ku, Kyoto 606-8502, Japan.}
\email{ishidat@math.kyoto-u.ac.jp}
\date{\today}

\begin{abstract}
We prove that 
the set of symmetrized conjugacy classes 
of the kernel of the Calabi homomorphism 
on the group of area-preserving diffeomorphisms of the $2$-disk 
is not quasi-isometric to the half line. 	
\end{abstract}

\maketitle
\section{Introduction}
Suppose that $G$ is a simple group 
and $K\subseteq G$ is a subset. 
Here, we assume that $K$ contains non-trivial elements of $G$. 
Since the group $G$ is simple, 
any non-trivial element $g$ of $G$ 
can be written as a product of conjugates of elements of $K\cup K^{-1}$. 
We define for each $g\in G$ the number $q_K(g)$ 
by the minimal number of conjugates of elements of $K\cup K^{-1}$ 
whose product is equal to $g$. 
Here, for the identity element $e$, we define $q_K(e)=0$. 
The function $q_K\colon G\to \Z _{\geq 0}$ is obviously invariant under conjugations 
and defines a conjugation-invariant norm on $G$. 
Such a conjugation-invariant norm 
is called a {\it conjugation-generated norm}. 
In this paper, we mainly consider 
the case $K$ consists 
of a single non-trivial element. 

Elements $f$ and $g$ of a group $G$ 
are {\it symmetrized conjugate} to each other 
if $f$ is conjugate to $g$ or $g^{-1}$. 
It is easy to see 
that symmetrized conjugacy is an equivalence relation. 
We denote by $[g]$ 
the symmetrized conjugacy class represented by $g\in G$. 
We define $\gmet{G}$ 
to be the set of non-trivial symmetrized conjugacy classes 
of elements of $G$. 
In \cite{tsuboi09}, 
Tsuboi introduced a metric $d$ on $\gmet{G}$ defined by 
\[ d([f], [g])=\log\max\{ q_{\{ g\}}(f), q_{\{ f\}}(g)\} . \]
In fact, 
it is easy to see that the inequality
\[ q_{\{ f\}}(h)\leq q_{\{ f\}}(g)q_{\{ g\}}(h) \]
holds for any $f, g, h\in G$ 
and thus the function $d\colon\gmet{G}\times\gmet{G}\to\R _{\geq 0}$ 
satisfies the triangle inequality. 
We are interested in this metric space $\gmet{G}$, 
which is an invariant of simple group. 

In \cite{kodama11p}, 
Kodama studied the metric space $(\gmet{G}, d)$ 
for the case $G$ is the infinite alternating group $A_{\infty}$ 
and proved the following. 

\begin{theorem}[Kodama \cite{kodama11p}]
The metric space $(\gmet{A_{\infty}}, d)$ is quasi-isometric to the half line. 
\end{theorem}

We define the $2$-disk $D^{2}$ 
and the standard area form $\Omega$ on $D^{2}$ to be 
\[ D^{2}=\{ (x, y)\in\R ;x^2+y^2\leq 1\}
\text{ and }\Omega =\frac{1}{\pi}dx\wedge dy \]
respectively. 
Let $\aread$ be the group of $C^\infty$-diffeomorphisms 
of the $2$-disk $D^2$,
which preserve $\Omega$ 
and are the identity on a neighborhood of the boundary. 
It is classically known 
that the group $\aread$ admits a homomorphism
\[ \Cal\colon\aread\to\R \]
called the Calabi homomorphism. 
The Calabi homomorphism gives an abelianization of $\aread$ 
and its kernel $\Ker\Cal$ is simple \cite{banyaga}. 
In this paper, 
we study the metric space $(\gmet{G}, d)$ 
for the case $G=\Ker\Cal$ 
and prove the following theorem. 

\begin{theorem}\label{main}
For any non-trivial element $f\in\Ker\Cal$, 
there exist a sequence $\{ f_{n}\}_{n\geq 0}$ contained in $\Ker\Cal$ 
with $f_{0}=f$,  
an element $g\in\Ker\Cal$ 
and positive constants $C_{1}, C_{2}, C_{3}$ 
which satisfy the following. 
\begin{enumerate}
\item[(i)] $d([f_{n}], [f_{m}])\geq C_{1}|n-m|$, 
\item[(ii)] $d([f_{n}], [f_{n+1}])\leq C_{2}$,  
\item[(iii)] $d([f_{n}], [g^{m}])\geq\log m+C_{3}$.  
\end{enumerate}
\end{theorem}

As a corollary, 
we obtain the following statement 
answering to a problem raised by Tsuboi \cite[Problem4.4]{tsuboiGF13}. 

\begin{theorem}\label{cor}
The metric space $(\gmet{\Ker\Cal}, d)$ is not quasi-isometric to the half line. 
\end{theorem}

\section{Quasi-morphisms}

In this section, 
we prepare a notion of quasi-morphism, 
which is a useful tool to evaluate a lower bound 
for a conjugation-generated norm $q_{K}$ 
and prove Proposition \ref{shrink}. 
On quasi-morphisms and conjugation-generated norms, 
see \cite{bip08} for more details. 

Let $G$ be a group. 
A {\it quasi-morphism} on $G$ is a function 
$\phi\colon G\to\R$ 
such that there exists a constant $C\geq 0$ 
and $|\phi (gh)-\phi (g)-\phi (h)|\leq C$ 
for any $g, h\in G$. 
The real number
\[ D(\phi ) =\sup _{g, h\in G} |\phi (gh)-\phi (g)-\phi (h)| \]
is called the {\it defect} of $\phi$.
A quasi-morphism $\phi$ on $G$ is {\it homogeneous} 
if $\phi (g^{p})=p\phi (g)$ for any $g\in G$ and any $p\in\Z$. 
For any quasi-morphism $\phi$ 
on an arbitrary group $G$, 
there exists a unique homogeneous quasi-morphism $\tilde{\phi}$ on $G$ 
such that $\tilde{\phi}-\phi$ is a bounded function on $G$ 
and $\tilde{\phi}$ is explicitly written as 
\[ \tilde{\phi}(g)
=\lim _{p\to\infty}\frac{1}{p}\phi (g^{p}). \] 
We denote by $Q(G)$ the $\R$-vector space 
consisting of homogeneous quasi-morphisms on $G$. 
Note that homogeneous quasi-morphisms 
are invariant under conjugations.

\subsection{Conjugation-invariant norms and quasi-morphisms}
Let  $K$ be a subset of $G$. 
We define the vector subspace $Q(G, K)$ of $Q(G)$ by
\[ Q(G, K)=\{\phi\in Q(G); \phi \text{ is bounded on } K \} . \]
Note that this definition is different from that given in \cite{bip08}. 
Suppose that $g\in G$ is written as 
\[ g=f_1\dots f_n, \]
where $f_1, \dots , f_n$ are conjugates of elements of $K\cup K^{-1}$.  
Then for $\phi\in Q(G, K)$ the inequation 
\[ |\phi (g)-\phi (f_1)-\dots -\phi (f_n)|\leq (n-1)(D(\phi )) \]
holds. 
If we set $C_{K}=\sup _{h\in K}|\phi (h)|$,  
then we have 
\[ \frac{|\phi (g)|}{D(\phi )+C_{K}}\leq n. \]
This means that 
\[ \frac{|\phi (g)|}{D(\phi )+C_{K}}\leq q_K(g). \]
Denoting by $[K]$ the set of symmetrized conjugacy classes 
represented by the elements of $K$, 
we have the following lemma 
on the metric $d$ of $\gmet{G}$.

\begin{lemma}\label{bound-power}
Let $\phi\in Q(G, K)$ and $g\in G$ such that $\phi (g)\neq 0$. 
Then 
\[ \log\frac{|\phi (g)|}{D(\phi )+C_{K}}\leq d([g], [K]). \]
In particular, 
\[ \log n+\log\frac{|\phi (g)|}{D(\phi )+C_{K}}\leq d([g^{n}], [K]) 
\text{ for any }n. \]
\end{lemma}

A simple group $G$ is {\it uniformly simple} 
if the metric space $(\gmet{G}, d)$ is bounded. 
This is equivalent to saying 
that $(\gmet{G}, d)$ is quasi-isometric to a point. 
Since $Q(G, K)=Q(K)$ for any bounded set $K$, 
if the group $G$ admits a non-trivial quasi-morphism 
then $(\gmet{G}, d)$ is unbounded by Lemma \ref{bound-power} 
and thus $G$ is not uniformly simple.

\subsection{Gambaudo-Ghys' construction of quasi-morphisms on $\aread$}

It is known that 
the vector space $Q(\aread )$ is infinite-dimensional 
\cite{ep03}\cite{gg04}\cite{ishida14}. 
To prove Theorem \ref{main}, 
we use quasi-morphisms on $\Ker\Cal$ 
obtained by Brandenbursky 
generalizing Gambaudo-Ghys' construction \cite{brandenbursky11}. 

Let $X_{n}(D^{2})$ be the $n$-fold configuration space of $D^{2}$. 
Fix a base point 
$x^{0}=(x_{1}^{0}, \dots , x_{n}^{0})\in X_{n}(D^{2})$. 
For any $g\in\aread$ 
and almost every $x=(x_{1}, \dots , x_{n})\in X_{n}(D^{2})$, 
we set a loop $l(g; x)\colon [0, 1]\to X_{n}(D^{2})$ by 
\[ l(g; x)(t)=\left\{ \begin{array}{ll}
\{ (1-3t)x_i^0+3tx_i\} &\displaystyle (0\leq t\leq\frac{1}{3}) \\[.6em]
\{ g_{3t-1}(x_i)\} &\displaystyle (\frac{1}{3}\leq t\leq\frac{2}{3})\\[.6em]
\{ (3-3t)g(x_i)+(3t-2)x_i^0\} &\displaystyle (\frac{2}{3}\leq t\leq 1), 
\end{array}\right. \]
where $\{ g_t\}_{t\in[0, 1]}$ is a path in $\aread$ 
such that $g_0$ is the identity and $g_1=g$. 
Of course for some $x\in X_{n}(D^{2})$ 
the loop $l(g; x)$ may not be defined. 
However, for almost every $x$ the loop $l(g; x)$ is well-defined.
We define the pure braid $\gamma (g; x)$ 
to be the homotopy class relative to the base point $x^{0}$ 
represented by the loop $l(g; x)$. 
Since the group of diffeomorphisms of $D^{2}$ 
is contractible \cite{smale59} 
and is homotopy equivalence to $\aread$ \cite{moser65}, 
the pure braid $\gamma (g; x)$ 
is independent of the choice of the path $\{ g_{t} \}$. 
Let $P_n(D^2)$ be the pure braid group on $n$-strands. 
For a homogeneous quasi-morphism $\phi$ on $P_{n}(D^{2})$, 
if we consider the function 
\[ g\mapsto \int _{x\in X_{n}(D^{2})}\phi (\gamma (g; x))\Omega ^{n}, \]
then this function is well-defined \cite{brandenbursky11}\cite{bk13}
and is a quasi-morphism on $\aread$
since the diffeomorphism $g$ preserves $\Omega$. 
Thus we have the linear map 
$\Gamma _{n}\colon Q(P_{n}(D^{2}))\to Q(\aread )$ 
defined by
\[ \Gamma _{n}(\phi )(g)
=\lim_{p\to\infty}\frac{1}{p}
\int _{x\in X_{n}(D^{2})}\phi (\gamma (g^{p}; x))\Omega ^{n}. \]

Let $B_{n}(D^{2})$ be the braid group on $n$ strands 
and $i\colon P_{n}(D^{2})\to B_{n}(D^{2})$ the natural inclusion. 
Then the linear map 
$Q(i)\colon Q(B_{n}(D^{2}))\to Q(P_{n}(D^{2}))$ 
is induced. 

For $r>1$, 
we denote by $D(r^{-1})$ the small disk   
$\{ (x, y)\in\R ; x^{2}+y^{2}\leq r^{-2}\}$ of radius $1/r$. 
Let $\varphi _{r}\colon D^{2}\to D(r^{-1})$ be the $C^{\infty}$-diffeomorphism 
defined by 
\[ \varphi _{r}(x, y)=\left(\frac{x}{r}, \frac{y}{r}\right). \]
We define the homomorphism 
$s_{r}\colon\aread\to\aread$ by 
\[ s_{r}(f)(x, y)=\left\{
\begin{array}{ll}
\varphi _{r}\circ f\circ\varphi _{r}^{-1}(x, y) & \text{if }(x, y)\in D(r^{-1}) \\ 
(x, y) & \text{if }(x, y)\notin D(r^{-1}) .
\end{array}
\right. \] 
Note that if $f$ is in $\Ker\Cal$, then $s_{r}(f)$ is also. 

Let $\sigma _{1}\in B_{3}(D^{2})$ be the braid on $3$ strands 
as indicated in Figure \ref{s1}.  
\begin{figure}[htbp]
\begin{center}
\unitlength 0.1in
\begin{picture}(  6.0000, 12.0000)( 16.0000,-34.0000)
%
\special{pn 20}%
\special{pa 1600 2200}%
\special{pa 2000 3400}%
\special{fp}%
%
\special{pn 20}%
\special{pa 2000 2200}%
\special{pa 1820 2746}%
\special{fp}%
%
\special{pn 20}%
\special{pa 1600 3400}%
\special{pa 1780 2860}%
\special{fp}%
%
\special{pn 20}%
\special{pa 2200 2200}%
\special{pa 2200 3400}%
\special{fp}%
\end{picture}%
\end{center}
\caption{the braid $\sigma _{1}$}
\label{s1}
\end{figure}
The following proposition is essentially introduced in \cite[Lemma 3.11]{bk13}. 

\begin{prop}\label{shrink}
If $\phi\in Q(B_{3})$ satisfies $\phi (\sigma _{1})=0$, 
then 
\[ \Gamma _{3}\circ Q(i)(\phi )(s_{r}(f))=
\frac{1}{r^6}\Gamma _{3}\circ Q(i)(\phi )(f) \] 
for any $f\in\aread$ and any $r>1$. 
\end{prop}

\begin{proof}
Let $x=(x_{1}, x_{2}, x_{3})$ be in $X_{3}(D^{2})$. 
For any $f\in\aread$ and any $r>1$, 
if two or three of $x_{1}, x_{2}, x_{3}$ are not in $D(r^{-1})$, 
then the pure braid $\gamma (s_{r}(f); x)$ is trivial. 
Hence we have 
\begin{align*}
\int _{x\in X_{3}(D^{2})}\phi (\gamma (s_{r}(f); x))\Omega ^{3}
&=\int _{x_{1}, x_{2}, x_{3}\in D(r^{-1})}\phi (\gamma (s_{r}(f); x))\Omega ^{3} \\
&\quad 
+3\int _{x_{1}, x_{2}\in D(r^{-1}), x_{3}\not\in D(r^{-1})}
\phi (\gamma (s_{r}(f); x))\Omega ^{3}
\end{align*}
for any $\phi\in Q(B_{3}(D^{2}))$. 

If $x_{1}, x_{2}\in D(r^{-1})$ and $x_{3}\not\in D(r^{-1})$, 
then the pure braid $\gamma (s_{r}(f); x)$ 
is a conjugate of a power of $\sigma _{1}$ 
and hence $\phi (\gamma (s_{r}(f); x))=0$. 
Since  
\[ \int _{x_{1}, x_{2}, x_{3}\in D(r^{-1})}\phi (\gamma (s_{r}(f); x))\Omega ^{3}
=\frac{1}{r^{6}}\int _{x\in X_{3}(D^{2})}\phi (\gamma (f; x))\Omega ^{3}, \]
we have the desired equality. 
\end{proof}

\section{Proof of the main theorem}

In this section, we prove the main theorem. 
Before starting the proof, 
we show the following lemma as a preliminary step. 
\begin{lemma}\label{area-contraction}
For any $f\in\aread$ and $r>1$, the following holds. 
\begin{enumerate}
\item[(I)] $d([s_{r}^{m}(f)], [s_{r}^{n}(f)])\geq (2\log r)|m-n|$ .
\item[(II)] $d([s_{r}^{n}(f)], [s_{r}^{n+1}(f)])\leq d([f], [s_{r}(f)])$.  
\end{enumerate}
\end{lemma}
\begin{proof}
Assume that $m<n$. 
Since the area of the support of $s_{r}^{m}(f)$ 
is just $r^{2(n-m)}$ times of that of $s_{r}^{n}(f)$, 
we have $q_{\{ s_{r}^{n}(f)\}}(s_{r}^{m}(f))\geq r^{2(n-m)}$. 
This implies (I). 

Suppose that $s_{r}(f)$ is written as a product 
\[ s_{r}(f)=(h_{1}f^{\varepsilon _{1}}h_{1}^{-1})\dots 
(h_{k}f^{\varepsilon _{k}}h_{k}^{-1}), \] 
where each $\varepsilon _{i}$ is $1$ or $-1$. 
Since the map $s_{r}\colon\aread\to\aread$ is a homomorphism, 
we have 
\[ 
s_{r}^{n+1}(f)=(s_{r}^{n}(h_{1})s_{r}^{n}(f)^{\varepsilon _{1}}s_{r}^{n}(h_{1})^{-1})\dots 
(s_{r}^{n}(h_{k})s_{r}^{n}(f)^{\varepsilon _{k}}s_{r}^{n}(h_{k})^{-1}) \]
and thus $q_{\{ s_{r}^{n}(f)\}}(s_{r}^{n+1}(f))\leq q_{\{ f\}}(s_{r}(f))$. 
Similarly the inequality $q_{\{ s_{r}^{n+1}(f)\}}(s_{r}^{n}(f))\leq q_{\{ s_{r}(f)\}}(f)$ 
also holds. 
Hence we have (II). 
\end{proof}
\begin{proof}[Proof of Theorem \ref{main}]
Fix $f\in\Ker\Cal$ and $r>1$. 
If we set $f_{n}=s_{r}^{n}(f)$, 
then the properties (i) and (ii) immediately follow 
from Lemma \ref{area-contraction}. 

Since the vector space $Q(B_{n}(D^{2}))$ is infinite-dimensional 
for $n\geq 3$ \cite{bf02}, 
considering the linear combination 
it is guaranteed  
that there exists a non-trivial homogeneous quasi-morphism $\phi$ on $B_{3}$
satisfying $\phi (\sigma _{1})=0$. 
Since the composition of the linear maps 
$\Gamma _{n}\circ Q(i)\colon Q(B_{n}(D^{2}))\to Q(\aread )$ 
is injective for $n\geq 3$ \cite{ishida14}, 
its image $\Gamma _{3}\circ Q(i)(\phi )$ 
is also non-trivial. 
We denote it by $\phi '$. 
By Proposition \ref{shrink}, 
$|\phi '(f_{n})|\leq |\phi '(f)|$ and 
thus $\phi '$ is in $Q(\aread , \{ f_{n}; n\geq 0\} )$. 
Moreover, choose $g\in\Ker\Cal$ such that $\phi '(g)\neq 0$. 
Then we have by Lemma \ref{bound-power}
\begin{align*}
\log m+\log\frac{|\phi '(g)|}{D(\phi ')+|\phi '(f)|}\leq d([g^{m}], [f_{n}; n\geq 0]) 
\text{ for any }m\in\N , 
\end{align*}
which is the property (iii). 
\end{proof}

\begin{proof}[Proof of Theorem \ref{cor}]
If the metric spaces $\gmet{\Ker\Cal}$ and $\R _{\geq 0}$ 
are quasi-isometric, 
then there exists 
a quasi-isometric embedding $\Phi\colon\gmet{\Ker\Cal}\to\R _{\geq 0}$. 
By the property (iii), 
we have $\Phi ([f]) <\Phi ([g^{m}])$ 
for sufficiently large $m\in\N$. 
By the property (i), 
there exists $n\in\N$ such that $\Phi ([g^{m}])<\Phi ([f_{n}])$. 
If we set $n_{m}=\min\{ n\in\N ; \Phi ([g^{m}])<\Phi ([f_{n}])\}$,  
then $\Phi  ([f_{n_{m}}])-\Phi ([g^{m}])$ is bounded independently on $m$ 
by the property (ii). 
However 
this contradicts the property (iii) 
since we can make $n_{m}$ arbitrarily large 
by taking larger $m$. 
\end{proof}

\begin{remark}
Let $M$ be a closed $C^{\infty}$-manifold 
and fix a symplectic form $\omega$ of $M$. 
Then the group $\hamm$ 
of Hamiltonian diffeomorphisms of $M$ 
is a simple group \cite{banyaga}. 

Let $U$ be a closed ball in $M$. 
Taking the subgroup $\hamu$ of $\hamm$, 
consisting of diffeomorphisms supported by $U$, 
as in the case of $D^{2}$ 
we can consider the shrinking homomorphism 
$s_{r}\colon\hamu\to\hamu$ 
and construct a sequence $\{f _{n}\}$ in $\hamm$ 
which satisfies the properties (i) and (ii) in Theorem \ref{main}. 
Hence if there exists a quasi-morphisms on $\hamm$ 
whose restriction in $\hamu$ have the property as Proposition \ref{shrink}, 
then Theorem \ref{main} holds for $\hamm$ 
and Theorem \ref{cor} for $\mathcal{M}(\hamm )$. 

When $M$ is a closed surface, 
we can construct quasi-morphisms on $\volm$ 
by Gambaudo-Ghys' way \cite{brandenbursky15p}
and verify by an argument similar to the case of $D^{2}$ 
that there exists a quasi-morphism $\phi$ on $\hamm$
satisfying $\phi (s_{r}(f))=r^{-6}\phi (f)$ for any $f\in\hamu$. 

When $M$ is the one point blow up 
of a closed symplectic $4$-manifold $(X, \omega _{X})$ 
such that $\omega _{X}$ and the first Chern class $c_{1}(X)$ 
vanish on $\pi _{2}(X)$, 
then $\hamm$ admits a non-trivial quasi-morphism $\mu$, 
which is called a Calabi quasi-morphism \cite{ep03}\cite{mcduff10}. 
If we take $U$ sufficiently small, 
then  $\mu$ satisfies 
$\mu (s_{r}(f))=r^{-8}\mu (f)$ for any $f\in\hamu$. 
\end{remark}

\begin{remark}
Let $\hamd$ and $\hamr$ be the groups of 
Hamiltonian diffeomorphisms of $D^{2n}$ and $\R ^{2n}$ respectively 
with respect to the standard symplectic form $\omega$. 
These groups admits the Calabi homomorphisms 
$\Cal\colon\hamd\to\R$ and $\Cal _{\R}\colon\hamr\to\R$ 
and their kernels $\Ker\Cal$ and $\Ker\Cal _{\R}$ are simple \cite{banyaga}. 
The group $\hamd$ admits a quasi-morphism $\tau$, 
which is constructed by Barge and Ghys \cite{bargeghys92}. 
The quasi-morphism $\tau\in Q(\hamd)$ satisfies 
$\tau (s_{r}(f))=r^{-2n}(f)$. 

Although the group $\Ker\Cal _{\R}$ 
does not admit non-trivial quasi-morphisms \cite{kotschick08}, 
Kawasaki constructed 
a homogeneous conjugation invariant function on $\Ker\Cal _{\R}$, 
which is called a partial quasi-morphism \cite{kawasaki14p}.  
If we denote it by $\mu$, 
then the equation $\mu (s_{r}(f))={r^{-2n}}\mu (f)$ is satisfied. 

Therefore a statement similar to Lemma \ref{bound-power} 
hold for $\tau$ and $\mu$. 
Hence Theorem \ref{main} holds for $\Ker\Cal$ and $\Ker\Cal _{\R}$ 
and Theorem \ref{cor} 
for $\mathcal{M}(\Ker\Cal)$ and $\mathcal{M}(\Ker\Cal _{\R})$. 
\end{remark}

\vskip 5pt
\noindent \textbf{Acknowledgments.}
The author wishes to express his gratitude 
to Jarek K\c{e}dra and Morimichi Kawasaki 
for reading the manuscript and several comments. 
The author is supported by JSPS Research Fellowships
for Young Scientists (26$\cdot$110). 

\providecommand{\bysame}{\leavevmode\hbox to3em{\hrulefill}\thinspace}
\providecommand{\MR}{\relax\ifhmode\unskip\space\fi MR }
\providecommand{\MRhref}[2]{%
  \href{http://www.ams.org/mathscinet-getitem?mr=#1}{#2}
}
\providecommand{\href}[2]{#2}


\begin{thebibliography}{10}

\bibitem{banyaga}
A.~Banyaga, \emph{The structure of classical diffeomorphism groups},
  Mathematics and its Applications, vol. 400, Kluwer Academic Publishers Group,
  Dordrecht, 1997. \MR{1445290 (98h:22024)}

\bibitem{bargeghys92}
J.~Barge and {\'E}.~Ghys, \emph{Cocycles d'{E}uler et de {M}aslov}, Math. Ann.
  \textbf{294} (1992), no.~2, 235--265. \MR{1183404}

\bibitem{bf02}
M.~Bestvina and K.~Fujiwara, \emph{Bounded cohomology of subgroups of mapping
  class groups}, Geom. Topol. \textbf{6} (2002), 69--89 (electronic).
  \MR{1914565 (2003f:57003)}

\bibitem{brandenbursky15p}
M.~Brandenbursky, \emph{Bi-invariant metrics and quasi-morphisms on groups of
  hamiltonian diffeomorphisms of surfaces}, preprint, to appear in Internat. J.
  Math.

\bibitem{brandenbursky11}
\bysame, \emph{On quasi-morphisms from knot and braid invariants}, J. Knot
  Theory Ramifications \textbf{20} (2011), no.~10, 1397--1417. \MR{2851716
  (2012i:57002)}

\bibitem{bk13}
M.~Brandenbursky and J.~K\c{e}dra, \emph{On the autonomous metric on the group
  of area-preserving diffeomorphisms of the 2--disc}, Algebr. Geom. Topol.
  \textbf{13} (2013), no.~2, 795--816. \MR{3044593}

\bibitem{bip08}
D.~Burago, S.~Ivanov, and L.~Polterovich, \emph{Conjugation-invariant norms on
  groups of geometric origin}, Groups of diffeomorphisms, Adv. Stud. Pure
  Math., vol.~52, Math. Soc. Japan, Tokyo, 2008, pp.~221--250. \MR{2509711
  (2011c:20074)}

\bibitem{ep03}
M.~Entov and L.~Polterovich, \emph{Calabi quasimorphism and quantum homology},
  Int. Math. Res. Not. (2003), no.~30, 1635--1676. \MR{1979584 (2004e:53131)}

\bibitem{gg04}
J-.~M. Gambaudo and {\'E}.~Ghys, \emph{Commutators and diffeomorphisms of
  surfaces}, Ergodic Theory Dynam. Systems \textbf{24} (2004), no.~5,
  1591--1617. \MR{2104597 (2006d:37071)}

\bibitem{ishida14}
T.~Ishida, \emph{Quasi-morphisms on the group of area-preserving
  diffeomorphisms of the 2-disk via braid groups}, Proc. Amer. Math. Soc. Ser.
  B \textbf{1} (2014), 43--51. \MR{3181631}

\bibitem{kawasaki14p}
M.~Kawasaki, \emph{Relative quasimorphisms and stably unbounded norms on the
  group of symplectomorphisms of the euclidean spaces}, preprint, to appear in
  J. Symplectic Geom.

\bibitem{kodama11p}
H.~Kodama, \emph{On non-uniformly simple groups}, preprint, arXiv:1107.5125.

\bibitem{kotschick08}
D.~Kotschick, \emph{Stable length in stable groups}, Groups of diffeomorphisms,
  Adv. Stud. Pure Math., vol.~52, Math. Soc. Japan, Tokyo, 2008, pp.~401--413.
  \MR{2509718 (2011j:57003)}

\bibitem{mcduff10}
D.~McDuff, \emph{Monodromy in {H}amiltonian {F}loer theory}, Comment. Math.
  Helv. \textbf{85} (2010), no.~1, 95--133. \MR{2563682 (2011d:53222)}

\bibitem{moser65}
J.~Moser, \emph{On the volume elements on a manifold}, Trans. Amer. Math. Soc.
  \textbf{120} (1965), 286--294. \MR{0182927 (32 \#409)}

\bibitem{smale59}
S.~Smale, \emph{Diffeomorphisms of the {$2$}-sphere}, Proc. Amer. Math. Soc.
  \textbf{10} (1959), 621--626. \MR{0112149 (22 \#3004)}

\bibitem{tsuboi09}
T.~Tsuboi, \emph{On the uniform simplicity of diffeomorphism groups},
  Differential geometry, World Sci. Publ., Hackensack, NJ, 2009, pp.~43--55.
  \MR{2523489 (2010k:57062)}

\bibitem{tsuboiGF13}
\bysame, \emph{Several problems on groups of diffeomorphisms}, Geometry and
  Foliations 2013, 2013.

\end{thebibliography}
\end{document}